\documentclass[reqno]{amsart}
\usepackage{amsmath}
\usepackage{amsthm, amscd, amssymb, amsfonts, amsbsy}
\usepackage[usenames, dvipsnames]{color}
\usepackage{enumerate}
\usepackage{esint}
\usepackage{mathrsfs}
\usepackage{todonotes}
\usepackage[colorlinks=true, linkcolor=blue, citecolor=blue]{hyperref}
\usepackage{verbatim}
\usepackage{kotex}
\usepackage{cite}
\usepackage{float}

\usepackage{tabularx}
\usepackage{booktabs}
\usepackage{ragged2e}

\numberwithin{equation}{section}

\theoremstyle{plain}
\newtheorem{theorem}{Theorem}[section]

\newtheorem{proposition}[theorem]{Proposition}

\theoremstyle{definition}

\newtheorem{remark}[theorem]{Remark}

\def\bR{\mathbb{R}}
\def\bH{\mathbb{H}}

\def\cH{\mathcal{H}}
\def\cQ{\mathcal{Q}}

\def\cK{\mathcal{K}}

\begin{document}
\title[]{Failure of zero extension in parabolic Sobolev spaces}

\author[J. Choi]{Jongkeun Choi}
\address[J. Choi]{Department of Mathematics Education, Pusan National University,  Busan, 46241, Republic of Korea}
\email{\href{mailto:jongkeun_choi@pusan.ac.kr}{\nolinkurl{jongkeun_choi@pusan.ac.kr}}}
\thanks{J. Choi was supported by the National Research Foundation of Korea (NRF) grant funded by the Korea government (MSIT) (RS-2026-25477708).}

\author[D. Kim]{Doyoon Kim}
\address[D. Kim]{Department of Mathematics, Korea University, 145 Anam-ro, Seongbuk-gu, Seoul, 02841, Republic of Korea}
\email{\href{mailto:doyoon_kim@korea.ac.kr}{\nolinkurl{doyoon_kim@korea.ac.kr}}}
\thanks{D. Kim was supported by the National Research Foundation of Korea (NRF) grant funded by the Korea government (MSIT) (RS-2025-16065192) and by a Korea University Grant.}

\author[K. Woo]{Kwan Woo}
\address[K. Woo]{Departement Mathematik und Informatik, Universit\"at Basel, CH-4051 Basel, Switzerland}
\email{\href{mailto:kwan.woo@unibas.ch}{\nolinkurl{kwan.woo@unibas.ch}}}
\thanks{K. Woo was supported by SNF Project 212573 FLUTURA – Fluids, Turbulence, Advection.}


\begin{abstract}
We show that spatial zero extension across the boundary may fail in parabolic Sobolev spaces $\mathring{\cH}^1_p((0,T) \times \Omega)$, which can also be characterized as
$$
L_p(0,T;\mathring{W}^1_p(\Omega))\cap W^1_p(0,T; W^{-1}_{p}(\Omega)).
$$
More precisely, for any $p\in [1, \infty)$, we construct a function $u\in \mathring{\cH}^1_p((0,T)\times \bR^d_+)$ whose zero extension does not belong to $\cH^1_p((0,T)\times \bR^d)$.
The obstruction occurs even for a flat boundary and is caused by a self-similar boundary layer concentrated at the initial-boundary corner, which produces a boundary supported normal flux defect after zero extension.
We also discuss the suitability of various Sobolev-type spaces as solution spaces for parabolic equations in divergence form.
\end{abstract}

\maketitle

%========================================
\section{Introduction and main results}	\label{S1}
%========================================

It is well known that the zero extension of a function $u \in \mathring{W}_p^1(\Omega)$, where $\Omega \subset \bR^d$ is an open set and $\mathring{W}_p^1(\Omega)$ denotes the closure of $C_0^\infty(\Omega)$ in $W_p^1(\Omega)$, belongs to $W_p^1(\bR^d)$.
This property arises naturally in various fields of analysis and is highly useful in the theory of PDEs.
Specifically, it allows one to reduce problems posed on the domain $\Omega$ to corresponding problems on the whole space, thereby enabling the use of global tools---for example, the Sobolev--Gagliardo--Nirenberg embedding theorem and Fourier analysis techniques on $\bR^d$.
Such reductions are widely used in the study of Dirichlet problems for divergence-form elliptic equations, particularly in establishing solvability and regularity results.
A natural question then arises: 

\medskip

\begin{equation}
\tag{Q}
\label{main_question}
\begin{minipage}{0.85\textwidth}
\itshape
Does the extension property hold for functions $u(t,x)$ arising in parabolic equations and having zero trace on the spatial boundary?
\end{minipage}
\end{equation}

\medskip

In this paper, we study this problem in standard parabolic Sobolev spaces for divergence form equations.
More precisely, we investigate whether functions vanishing on the lateral boundary remain in the same Sobolev class after extension by zero to the whole space.
From this perspective, we also discuss the \emph{suitability} of various Sobolev-type spaces for the solvability of parabolic equations.

The parabolic spaces mainly considered in this paper are 
$$
\cH^1_p((0,T)\times \Omega),  \quad \mathring{\cH}^1_p((0,T)\times \Omega),
$$
where $p\in [1, \infty)$, $T\in (0, \infty]$, and $\Omega$ is a domain in $\bR^d$.
The space $\cH^1_p$ consists of functions $u$ with $u, Du\in L_p$, while the time derivative admits the distributional form
\begin{equation}		\label{260226_eq1}
u_t=D_i g_i+f
\end{equation}
for some $g_i, \, f\in L_p$.
The space $\mathring{\cH}^1_p$ corresponds to the zero-lateral-boundary subspace of $\cH^1_p$.
These spaces can be characterized in terms of Bochner--Sobolev spaces.
For instance, for $p\in (1, \infty)$, the following equivalence holds: 
$$
\mathring{\cH}^1_p((0,T)\times \Omega) \sim L_p(0,T; \mathring{W}^1_p(\Omega))\cap W^1_p(0,T; W^{-1}_p(\Omega)).
$$
See \S \ref{S1_1} for precise definitions of $\cH^1_p$ and $\mathring{\cH}^1_p$, as well as a brief discussion of this equivalence.

These spaces provide a natural functional framework for second-order parabolic equations in divergence form, such as linear equations of the form
\[
u_t-D_i(a^{ij}D_j u+a^iu)+b^iD_iu+cu=D_i g_i+f.
\]
Indeed, when all coefficients are bounded, the time derivative $u_t$ can be interpreted in the sense of \eqref{260226_eq1}.
In particular, $\mathring{\cH}^1_p$ is well suited for Dirichlet problems with zero lateral boundary conditions, since the boundary condition is built into the definition of the space.
Motivated by advances in parabolic theory, this Sobolev framework has been widely used in the study of $L_p$-estimates and solvability for divergence-form equations; see, for instance, \cite{MR2541414, MR2187159, MR2329320, MR2832162, MR2835999, MR2304157} for results on both linear and nonlinear equations, and \cite{MR3947859, MR2650802, MR3812104, arXiv:2308.09220, MR4704640, MR2352490} for further results in more general functional settings involving mixed and/or weighted norms.
Some of these works address higher-order, kinetic, and nonstationary Stokes equations.
We also refer to \cite{MR4345837, MR4920684} for related results in the spaces with fractional time derivatives.
For further discussions on Bochner--Sobolev type function spaces such as
\[
L_p(0,T; X_1)\cap W^1_p(0,T; X_0),
\]
see, for instance, \cite{MR1345385, MR4249416, MR1230384, MR2597943, MR0350177, MR3012216, MR3524106, MR1846644} and the references therein.

\bigskip

The main purpose of this paper is to show that, in general, the answer to \eqref{main_question} is negative.
More precisely, we prove that the zero extension may fail in $\mathring{\cH}^1_p$ for every $p\in [1, \infty)$, even when the underlying domain is the half-space
$$
\bR^d_+=\{x=(x_1, x')\in \bR^d: x_1>0, \, x'\in \bR^{d-1}\},
$$
which has a \emph{flat} boundary.
The main result is the following.

\begin{theorem}		\label{MT}
Let $T\in (0, \infty]$.
For any $p\in [1, \infty)$,  there exists a function $u\in \mathring{\cH}^1_{p}((0,T)\times \bR^d_+)$ such that 
$$
\bar{u}\not\in \cH^1_{p}((0,T)\times \bR^d),
$$
where $\bar{u}$ is the zero extension of $u$ with respect to the spatial variable $x$.
More precisely, although $\bar u, D\bar u\in L_p((0,T)\times\bR^d)$, there do not exist
\[
G_i, \,F\in L_p((0,T)\times\bR^d), \quad i=1,\ldots,d
\]
such that
\[
\bar u_t=D_iG_i+F
\quad \text{ in }\,\, (0,T)\times\bR^d.
\]
\end{theorem}

\begin{remark}
A few remarks regarding Theorem~\ref{MT} are in order.
\begin{enumerate}[(i)]
    \item The proof of Theorem~\ref{MT} actually yields a slightly stronger statement.
Namely, for every $p\in[1,\infty)$, the representation
\[
\bar u_t = D_iG_i+F
\]
considered above is impossible even under the weaker requirement $G_i,F\in L_{p,\operatorname{loc}}([0,T)\times \bR^d)$, that is,
\[
G_i,F\in L_p((0,\tau)\times K), \quad i=1,\ldots,d
\]
for every $\tau \in (0,T)$ and every compact set $K \subset \bR^d$.

\item 
Unlike the case $p \in [1, \infty)$, the zero extension $\bar{u}$ of any $u \in \mathring{\cH}^1_\infty((0,T) \times \bR^d_+)$ belongs to $\cH^1_\infty((0,T) \times \bR^d)$.
This property relies on a duality argument and extends to Lipschitz domains via local flattening.
See \S \ref{Sec:p_infty} for the detailed proof and further discussion.

\end{enumerate}
\end{remark}

It has been previously pointed out that spatial zero extension need not preserve the parabolic Sobolev class $\mathring{\cH}^1_p$ in the above sense; see, for example, \cite{MR4387945}.
However, to the best of the authors' knowledge, an explicit construction or detailed justification cannot be found in the existing literature.
In this paper, we provide an explicit construction and a rigorous proof of this failure.

\bigskip

The remainder of the paper is organized as follows.
The rest of \S\ref{S1} is devoted to preparatory material, generalizations, and implications of the main result.
In particular, \S\ref{S1_2} explains the main mechanism behind the proof of Theorem \ref{MT}.
We also discuss implications of Theorem~\ref{MT} in \S\ref{S1_4}, and summarize the comparison among the relevant function spaces in \S\ref{S1_5}.
In \S\ref{S2}, we provide the detailed proof of the main theorem and treat the case $p=\infty$.

%========================================
\subsection{Function spaces}		\label{S1_1}
%========================================

Let $T\in (0, \infty]$ and $\Omega$ be a domain in $\bR^d$, where $d\ge 1$.
For $p\in [1, \infty]$, we define 
$$
W^{1,0}_p((0,T)\times \Omega)=\{u: D^\ell_x u\in L_p((0,T)\times \Omega), \, |\ell|\le 1 \}
$$
with the norm
$$
\|u\|_{W^{1,0}_p((0,T)\times \Omega)}=\|u\|_{L_p((0,T)\times \Omega)}+\|Du\|_{L_p((0,T)\times \Omega)}.
$$
By $u_t\in \bH^{-1}_p((0,T)\times \Omega)$, we mean that there exist $g=(g_1,\ldots, g_d)\in L_p((0,T)\times \Omega)^d$ and $f\in L_{p}((0,T)\times \Omega)$ such that 
$$
u_t=D_i g_i+f \quad \text{in }\, (0,T)\times \Omega
$$
in a distribution sense, and that
$$
\|u_t\|_{\bH^{-1}_{p}((0,T)\times \Omega)}=\inf\big\{\|g\|_{L_{p}((0,T)\times \Omega)}+\|f\|_{L_{p}((0,T)\times \Omega)}: u_t=D_i g_i+f\big\}
$$
is finite.

We define $\cH^1_p((0,T)\times \Omega)$ and $V_p((0,T)\times \Omega)$ as the spaces of all functions in $W^{1,0}_{p}((0,T)\times \Omega)$ equipped with the norms
$$
\|u\|_{\cH^1_{p}((0,T)\times \Omega)}=\|u_t\|_{\bH^{-1}_{p}((0,T)\times \Omega)}+\|u\|_{W^{1,0}_p((0,T)\times \Omega)}
$$
and 
$$
\|u\|_{V_p((0,T)\times \Omega)}=\operatorname*{ess\,sup}_{t\in (0, T)} \|u(t, \cdot)\|_{L_p(\Omega)}+\|u\|_{W^{1,0}_{p}((0,T)\times \Omega)},
$$
respectively.

Let $C^\infty_0([0,T]\times \Omega)$ be the set of all infinitely differentiable functions in $[0,T]\times \Omega$ with compact support in $[0,T]\times \Omega$,
where if  $T=\infty$, the interval $[0, T]$ is understood as $[0, \infty)$.
We denote by $\mathring{\cH}^1_p((0,T)\times \Omega)$ and $\mathring{V}_p((0,T)\times \Omega)$ the closures of $C^\infty_0([0,T]\times \Omega)$ in $\cH^1_p((0,T)\times \Omega)$ and $V_p((0,T)\times \Omega)$, respectively.

The corresponding mixed-norm spaces $W^{1,0}_{p,q}$, $\cH^1_{p,q}$, $V_{p,q}$, and so on, are defined analogously, replacing all $L_p$-norms in the preceding definitions by the following $L_{p,q}$-norm:
$$
\|u\|_{L_{p,q}((0,T)\times \Omega)}=\bigg(\int_0^T \bigg(\int_{\Omega}|u(t,x)|^p\,dx\bigg)^{q/p}\,dt\bigg)^{1/q}.
$$
As usual, the definition of the norm is modified in the standard way when $p=\infty$ or $q=\infty$.

\begin{remark}		\label{RMK_0}
The parabolic Sobolev spaces introduced above can be described in terms of Bochner--Sobolev spaces.
More precisely, for $p\in (1, \infty)$, the following identifications hold with equivalent norms:
\begin{equation}		\label{260506_eq1}
\cH^1_p((0,T)\times \Omega) \sim L_p(0,T;W^1_p(\Omega))\cap W^1_p(0,T; W^{-1}_{p}(\Omega))
\end{equation}
and 
$$
V_p((0,T)\times \Omega) \sim L_p(0,T;W^1_p(\Omega))\cap L_\infty(0,T; L_p(\Omega)).
$$
Here, $W^{-1}_{p}(\Omega)$ denotes the dual space of $\mathring{W}^{1}_{p'}(\Omega)$, where $\mathring{W}^{1}_{p'}(\Omega)$ is the closure of $C^\infty_0(\Omega)$ in $W^1_{p'}(\Omega)$  and $p'$ is the conjugate exponent of $p$.
For the corresponding spaces with zero lateral boundary conditions, the spatial Sobolev space $W^1_p(\Omega)$ is replaced by $\mathring W^1_p(\Omega)$ 
yielding
\[
\mathring{\cH}_p^1\left((0,T) \times \Omega\right) \sim L_p(0,T; \mathring{W}_p^1(\Omega)) \cap W_p^1 (0,T; W_p^{-1}(\Omega)).
\]

For the reader's convenience, we give a brief justification of  \eqref{260506_eq1} by verifying, in particular, the norm equivalence
$$
\|u_t\|_{\bH^{-1}_p(\cQ)} \sim \|u_t\|_{L_p(0,T; W^{-1}_p(\Omega))},
$$
where $\cQ=(0,T)\times \Omega$.
Let 
$$
u\in \cH^1_p(\cQ),
$$
and suppose that
$$
u_t=D_i g_i+f \quad \text{in }\, \cQ
$$
for some $g_i,\,f\in L_p((0,T)\times \Omega)$.
Then, it is easy to see that $u \in L_p(0,T; W_p^1(\Omega))$.
Set
\[
\langle F(t), \varphi \rangle = \int_\Omega \left(f(t,x) \varphi(x) - g_i(t,x) D_i\varphi(x) \right) \, dx \quad \text{for} \,\, \varphi \in \mathring{W}_{p'}^1(\Omega).
\]
This is well defined for a.e. $t \in [0,T]$,
and $F$ belongs to $L_p(0,T; W_p^{-1}(\Omega))$ with
\[
\|F\|_{L_p(0,T; W_p^{-1}(\Omega))} \leq \|f\|_{L_p(\cQ)}+\|g_i\|_{L_p(\cQ)}.
\]
One can also verify that $u_t = F$ in the sense of the Bochner weak derivative.
Hence, we see that $u \in L_p(0,T;W^1_p(\Omega))\cap W^1_p(0,T; W^{-1}_{p}(\Omega))$ and
$$
\|u_t\|_{L_p(0,T; W^{-1}_p(\Omega))}\lesssim \|u_t\|_{\bH^{-1}_p(\cQ)}.
$$

Conversely, suppose that 
$$
u\in L_p(0,T;W^1_p(\Omega))\cap W^1_p(0,T; W^{-1}_{p}(\Omega)).
$$
By standard properties of Bochner spaces, $u$ admits a jointly measurable representative on $\cQ$, still denoted by $u$, with $u, Du \in L_p(\cQ)$.
Consider a linear functional $\ell:X\to \bR$ defined by 
\begin{equation}		\label{260508_eq1}
\ell(-D_1\phi,\ldots, -D_d\phi, \phi)=\int_0^T \langle u_t(t), \phi(t, \cdot)\rangle \,dt,
\end{equation}
where
$$
X=\{(-D_1\phi, \ldots, -D_d\phi, \phi) \in L_{p'}(\cQ)^{d+1}: \phi\in L_{p'}(0,T;\mathring{W}^1_{p'}(\Omega))\}.
$$
Since $X$ is a subspace of $L_{p'}(\cQ)^{d+1}$ and $\ell$ is bounded with
$$
\|\ell\|_{X^*}\lesssim \|u_t\|_{L_p(0,T; W^{-1}_p(\Omega))}, 
$$
the Hahn-Banach theorem allows us to extend $\ell$ to a bounded linear functional on $L_{p'}(\cQ)^{d+1}$.
Moreover, by the duality of $L_{p'}(\cQ)^{d+1}$, there exists
$$
(G_1, \ldots, G_d, F)\in L_p(\cQ)^{d+1}
$$
such that for all $(\eta_1,\ldots,\eta_d, \zeta)\in L_{p'}(\cQ)^{d+1}$,
\begin{equation}		\label{260508_eq2}
\ell(\eta_1,\ldots, \eta_d, \zeta)=\int_{0}^T\int_{\Omega} (G_i \eta_i+F\zeta)\,dx\,dt, 
\end{equation}
and that 
$$
\|G_i\|_{L_p(\cQ)}+\|F\|_{L_p(\cQ)}\lesssim \|u_t\|_{L_p(0,T;W^{-1}_p(\Omega))}.
$$
Combining \eqref{260508_eq1} and \eqref{260508_eq2}, we have 
$$
\int_0^T \langle u_t(t), \phi(t, \cdot)\rangle \,dt=\int_{0}^T\int_{\Omega} (-G_i D_i \phi+F \phi)\,dx\,dt
$$
for all $\phi\in C^\infty_0(\cQ)$.
Consequently, 
$$
u_t=D_iG_i+F \quad\text{in }\, \cQ
$$
and
$$
\|u_t\|_{\mathbb H^{-1}_p(\cQ)}\lesssim\|u_t\|_{L_p(0,T;W^{-1}_p(\Omega))}.
$$
This proves the converse inequality.
\end{remark}

%========================================
\subsection{Idea of the proof}		\label{S1_2}
%========================================

We illustrate the mechanism behind Theorem \ref{MT} in the special case $p=d=1$.
The argument for $p>1$ requires a different implementation, but the basic idea is the same, and it extends straightforwardly to higher dimensions $d\ge 2$.

Consider a function on $(t, x) \in (0,1)\times\bR_+$ of the form
\[
u(t,x)=\partial_x v(t,x), \quad \textrm{ where }\,\,
v(t,x)=t^{-1/4}e^{-x^2/t}.
\]
Indeed,
\[
u(t,x)=t^{-3/4}U(x/\sqrt t), \quad v(t,x)=t^{-1/4}V(x/\sqrt t),
\]
where
\[
V(y)=e^{-y^2}, \quad U(y) = V'(y) = -2y e^{-y^2}.
\]
Thus, for small $t>0$, the profile is concentrated in the region $x \sim \sqrt t$ near $x=0$.
A direct computation shows that
\[
u,\; \partial_x u \in L_1((0,1) \times \bR_+),
\quad
u_t\notin L_1((0,1)\times\bR_+),
\]
and that the time derivative admits the special structure
\[
u_t=\partial_x g,
\quad
g=v_t\in L_1((0,1)\times\bR_+).
\]
The relevant point is how the spatial norms of these functions scale as $t \searrow 0$.
Indeed, one finds
\[
\|u(t, \cdot)\|_{L_1(\bR_+)} \sim t^{-1/4},
\quad
\|u_x(t, \cdot)\|_{L_1(\bR_+)}, \, \|g(t, \cdot)\|_{L_1(\bR_+)} \sim t^{-3/4},
\]
while
\[
\|u_t(t, \cdot)\|_{L_1(\bR_+)} \sim t^{-5/4}.
\]
That is, the time derivative $u_t$ exhibits a non-integrable temporal singularity, while $u, u_x$, and $g$ remain relatively mild.
On the half-line, however, this wild behavior is completely absorbed by the divergence structure $u_t=\partial_x g$.
Thus, the singularity is harmless at the distributional level, and since $u(t, 0) = 0$ we have
\[
u\in \mathring{\cH}^1_1((0,1) \times \bR_+).
\]

The obstruction appears after extending $u$ by zero across the boundary $x=0$.
Formally, we have
\[
\bar{u}_t = \partial_x(\mathbf 1_{\{x>0\}} g)-g(t,0)\delta_0 \quad \textrm{ in } \,\, (0, 1) \times \bR,
\]
where $g(t,0)= -t^{-5/4}/4$ which is not integrable near $t=0$.
Thus spatial zero extension creates a boundary supported flux defect.
More precisely, assume for contradiction that the zero extension $\bar u$ still belongs to $\cH^1_1((0,1)\times\bR)$ and suppose, for simplicity, that $\bar u_t=\partial_x G$ (i.e. $F \equiv 0$) for some $G\in L_{1}((0,1)\times\bR)$.
Then $\bar u_t$ would satisfy the global identity
\begin{equation}
    \label{int_eq_whole}
\bigg( \int_0^1 \int_{\bR_+} u \varphi_t \,dx\,dt  = \bigg)
 \int_0^1 \int_{\bR} \bar u\,\varphi_t\,dx\,dt
=
\int_0^1 \int_{\bR} G\,\varphi_x\,dx\,dt
\end{equation}
for any $\varphi\in C^\infty_0((0,1)\times\bR)$.

Although the extension is performed only with respect to the spatial variable, \eqref{int_eq_whole} forces a compatibility between the divergence structure
and the temporal behavior of $u$.
To test this compatibility, one considers functions of the form
\[
\varphi(t,x)=\eta_m(t)\phi(x),
\]
where $\phi$ is supported near the boundary $x=0$ and $\eta_m$ is a temporal cutoff concentrating on an interval $[\tfrac{1}{2m}, \tfrac{1}{m}]$.
Plugging this choice into the left-hand side of
\eqref{int_eq_whole} yields the contribution
\begin{equation}
    \label{int_eq_div}
\int_{0}^{1} u(t,x)\,\eta_m'(t)\,dt \sim m \int_{1/(2m)}^{1/m} u(t,x)\,dt.
\end{equation}
Since $u(t, \cdot)$ is localized near the support of $\phi(\cdot)$, one finds that
\begin{equation}
    \label{int_eq_order}
 \|u(t, \cdot) \phi(\cdot)\|_{L_1(\bR_+)} \sim \|u(t, \cdot)\|_{L_1(\bR_+)} \sim t^{-1/4},
\end{equation}
and then, \eqref{int_eq_div} produces a divergent sequence
\[
m\int_{1/(2m)}^{1/m} t^{-1/4}\,dt \, \to \, \infty \quad \textrm{ as } \,\, m \to \infty
\]
on the left-hand side of \eqref{int_eq_whole}.
(Note that if $\phi$ vanishes near $x=0$, the first relation in \eqref{int_eq_order} no longer holds.)

On the other hand, the assumed representation of $\bar u_t$ by a $L_1$ function $G$ would force the right-hand side of \eqref{int_eq_whole} to remain bounded as $m \to \infty$.
This contradiction shows that the divergence structure of $u_t$, which is admissible on the half-line, cannot survive the zero extension.
In this sense, the failure of the extension is not just caused by a lack of spatial regularity, but also by the interaction between the spatial boundary and the temporal singularity.

\begin{remark}		\label{RMK_1}
Although $u(t,0)\equiv0$, the $L_1$-mass of $u(t,\cdot)$ is concentrated in the shrinking boundary layer $x\sim\sqrt t$ as $t\searrow0$.
Thus the singular behavior is located at the initial-boundary corner $(t,x)=(0,0)$.
\end{remark}

%========================================
\subsection{Generalizations}			\label{S1_3}
%========================================

We describe a generalization of Theorem \ref{MT} in two directions.

\begin{enumerate}
    \item Theorem~\ref{MT} extends to the mixed-norm Sobolev spaces $\mathring{\cH}^1_{p,q}((0,T)\times \bR^d_+)$, where  $p\in [1, \infty]$ and $q\in [1, \infty)$ with $(p,q)\neq (\infty, 1)$.
This follows by adapting the proof of the theorem (see \S \ref{S2}) with the condition
$$
\frac{1}{2p}+\frac{1}{q}-\frac{3}{2}\le a < \frac{1}{2p}+\frac{1}{q}-1
$$
in place of \eqref{240813_eq1}.

\item The result extends from the half-space to arbitrary Lipschitz domains $\Omega \subset \bR^d$.
This relies on a standard boundary flattening argument and the observation that a suitable spatial localization of the counterexample from Theorem \ref{MT}, when restricted near the boundary, still serves as a counterexample (see Remark \ref{RMK_1}).
Pulling this localized function back to the original domain $\Omega$ yields a function $v \in \mathring{\cH}^1_{p,q}((0,T)\times \Omega)$ whose zero extension does not belong to $\cH^1_{p,q}((0,T)\times \bR^d)$.
Moreover, there do not exist 
$$
G_i, \, F\in L_{p,q}((0,T)\times \bR^d), \quad i=1,\ldots, d
$$
such that 
$$
\bar{v}_t=D_iG_i+F \quad \textrm{in }\, (0,T)\times \bR^d,
$$
where $\bar{v}$ denotes the spatial zero extension of $v$.

\end{enumerate}

%========================================
\subsection{Some comments}	\label{S1_4}
%========================================

In this subsection, we discuss some implications of Theorem \ref{MT} concerning the suitability of Sobolev-type spaces as solution spaces for parabolic equations.

%========================================
\subsubsection{Realization as weak solutions}		\label{S1_4_1}
%========================================

Let $p\in (1, \infty)$ and $T\in (0, \infty]$.
We first recall that 
every function in $\cH^1_p((0,T)\times \bR^d)$ can be regarded as a weak solution to a parabolic equation in divergence form.
More precisely, 
if $u\in \cH^1_p((0,T)\times \bR^d)$ with
$$
u_t=D_iG_i+F
$$
for some $G_i,\, F\in L_p((0,T)\times \bR^d)$, then we can write
$$
u_t-\Delta u=D_i (G_i-D_i u)+F
$$
to see that $u$ satisfies an equation with the heat operator and $L_p$-data.
This observation remains valid when $\bR^d$ is replaced by an arbitrary domain.

We now contrast this with the space $V_p$.
Let $u\in \mathring{\cH}^1_p((0,T)\times \bR^d_+)$ be the counterexample from Theorem \ref{MT}, and let $\bar{u}$ denote its spatial zero extension.
From the estimate \eqref{260210_eq2} with $a=1/p-1<0$, it follows that 
$$
u\in \mathring{V}_p((0,T)\times \bR^d_+).
$$
Since the zero extension preserves the class $V_p$, we have
\[
\bar u\in V_p((0,T)\times \bR^d).
\]
However, Theorem \ref{MT} shows that $\bar u$ cannot be realized as a weak solution to \emph{any} parabolic equation in divergence form. 
Thus, unlike $\cH^1_p((0,T)\times \bR^d)$, 
the space $V_p((0,T)\times \bR^d)$ contains elements that do not arise as weak solutions to divergence-form equations.
The same phenomenon occurs when extending functions from an arbitrary Lipschitz domain $\Omega$.
This sharply contrasts with the elliptic setting:
for any domain $\Omega$, the zero extension of a function in $\mathring{W}^1_p(\Omega)$ trivially belongs to $W^1_p(\bR^d)$ and remains a weak solution to a divergence-form elliptic equation on the whole space.

%========================================
\subsubsection{Sobolev framework for parabolic equations}			\label{S1_4_2}
%========================================

The $L_p$-theory of parabolic equations is concerned with solvability and a priori estimates for solutions with data in Lebesgue spaces, under some regularity assumptions on the coefficients. 
In this theory, various types of Sobolev spaces have been employed as solution spaces, depending on the form of the equation, the singularity of the lower-order terms,  the boundary conditions under consideration, and the underlying norm structure.

Here, we focus on linear parabolic equations in divergence form with lower-order terms
\begin{equation}		\label{260223_eq1}
u_t-D_i(a^{ij}D_j u+a^iu)+b^iD_iu+cu=D_i g_i+f
\end{equation}
posed in a cylindrical domain $(0,T)\times \Omega$, and discuss the corresponding Sobolev spaces for the $L_p$-theory.

For equations with {\em bounded lower-order coefficients}, as noted earlier, the spaces $\cH^1_p$ and $\mathring{\cH}^1_p$ provide a natural setting for investigating unique solvability and estimates of the form
$$
\|u\|_{\cH^1_p}\lesssim \|g_i\|_{L_p}+\|f\|_{L_p},
$$
under zero initial conditions.
We again refer the reader to \cite{MR2187159, MR2329320, MR2832162, MR2835999, MR2304157} and the references therein for results in these Sobolev spaces under suitable assumptions on the leading coefficients $a^{ij}$ and on the boundary of the underlying domain.
Conversely, as mentioned in \S \ref{S1_4_1}, every function in $\cH^1_p$  can be regarded as  a weak solution to a parabolic equation of the form \eqref{260223_eq1}.
In this sense,  the Sobolev spaces $\cH^1_p$ and $\mathring{\cH}^1_p$ serve as natural and essentially optimal solution spaces for the $L_p$-solvability of parabolic equations with bounded lower-order coefficients.
In particular, $\mathring{\cH}^1_p$ is well suited for Dirichlet problems with zero lateral boundary conditions.
Similar observations remain valid in the mixed-norm spaces $\cH^1_{p,q}$ for equations with $L_{p,q}$-data.
%See \cite{MR2764911, MR2352490}.

By contrast, for parabolic equations with {\em (sub)critical lower-order coefficients}, namely
\begin{equation}		\label{260429_eq1}
a^i\in L_{\ell_1, r_1}, \quad  b^i\in L_{\ell_2, r_2}, \quad   c\in L_{\ell_3, r_3},
\end{equation}
the spaces $\cH^1_p$ as well as $\cH^1_{p,q}$ are, in general, no longer large enough to accommodate solutions; see \cite{MR4387945}.
In that work, the authors addressed this issue by establishing solvability results in a parabolic Sobolev framework based essentially on $W^{1,0}_{p,q}$ for $p,q\in (1, \infty)$, with the estimate
$$
\|u\|_{W^{1,0}_{p,q}}\lesssim \|g_i\|_{L_{p,q}}+\|f\|_{L_{p_1,q_1}},
$$
where the pairs $(\ell_k, r_k)$ and $(p_1, q_1)$ are determined by appropriate lower summability conditions.
When $p,q\in [2, \infty)$, they also employed the space $V_{p,q}$ in order to capture additional regularity of solutions with respect to the time variable.
We remark, however, that $V_{p,q}$ is not suitable for the solvability theory in the range $p,q\in (1,2)$; see \cite{MR2433518} for a counterexample.
Moreover, as mentioned in \S \ref{S1_4_1}, $V_{p,q}$ contains functions that cannot be realized as weak solutions to any parabolic equation.

Motivated by these observations, we finally discuss a related Sobolev space for divergence-form equations with an additional half-order time derivative term in the data.
More precisely, we consider parabolic equations of the form 
 \begin{equation}		\label{260223_eq5}
u_t-D_i(a^{ij}D_j u+a^iu)+b^iD_iu+cu=D_{t}^{1/2}h+D_i g_i+f
\end{equation}
in an {\em infinite cylinder} $\bR\times \Omega$, where the half-time derivative is defined as
$$
D_t^{1/2}h(t,x)=\frac{1}{\sqrt{8\pi}} \int_{\bR} \frac{h(t+\ell, x)-h(t,x)}{|\ell|^{3/2}}\,d\ell.
$$
This class of equations was treated in \cite{MR4920684, MR0200593}, mainly in the case $a^i=b^i=0$ with $c\ge 0$ constant, while the corresponding equations with bounded lower-order terms were briefly addressed.
There, the authors employed the Sobolev space $H^{1/2, 1}_p$, equipped with the norm
$$
\|u\|_{H^{1/2, 1}_p}:=\|D_t^{1/2}u\|_{L_p}+\|u\|_{W^{1,0}_p},
$$
to prove unique solvability
in the whole domain $\Omega=\bR^d$, under suitable regularity assumptions on $a^{ij}$.
Building on this approach, and investigating anisotropic embeddings on $H^{1/2,1}_p$ (cf. \cite[Appendix]{MR4387945}),
one may further consider the equations \eqref{260223_eq5} with unbounded lower-order terms as in \eqref{260429_eq1} in the spaces $H^{1/2,1}_p$ and $H^{1/2,1}_{p,q}$.

%========================================
\subsection{Summary}	\label{S1_5}
%========================================

In summary, the preceding discussion indicates that the Sobolev spaces $\cH^1_p$, $V_p$, and $H^{1/2,1}_p$, where $p\in (1, \infty)$, play rather different roles in the solvability theory of parabolic equations.
For the reader's convenience, we summarize the main features discussed above in Table \ref{summary}.

\begin{table}[htbp] 
\centering
\renewcommand{\arraystretch}{1.25}
\begin{tabularx}{\textwidth}{
|>{\arraybackslash}p{4cm}
|>{\arraybackslash}X
|>{\arraybackslash}X
|>{\arraybackslash}X|}
\hline
\multicolumn{1}{|c|}{\textbf{Property}}
& \multicolumn{1}{c|}{$\cH^1_p$}
& \multicolumn{1}{c|}{$V_p$}
& \multicolumn{1}{c|}{$H^{1/2,1}_p$}
\\
\hline

\textbf{Zero extension} \newline (vanishing on lateral boundary)
& Fails in general
& Holds
& Holds
\\
\hline

\textbf{Realizability as weak solutions}
& Holds
& Fails in general
& Holds$^\ast$
\\
\hline

\textbf{Solvability} \newline (bounded lower-order terms)
& Optimal
& Partial ($p \ge 2$)
& Optimal$^\ast$
\\
\hline

\textbf{Solvability} \newline (unbounded lower-order terms)
& Inadequate
& Partial ($p \ge 2$)
& Suitable$^\ast$ (to be discussed elsewhere)
\\
\hline

\end{tabularx}
\caption{
Comparison of Sobolev-type spaces. The assertions for $\cH^1_p$ and $V_p$ concern equation \eqref{260223_eq1}, while the superscript $^\ast$ refers to equation \eqref{260223_eq5} in $\bR\times \Omega$.
}
\label{summary}
\end{table}

%========================================
\section{Proof of the main theorem}			\label{S2}
%========================================

In this section, we prove our main result Theorem \ref{MT}.
It suffices to consider the theorem on a finite-height cylinder, and hence, by scaling, we may assume that $(0,T)=(0,1)$.

We first treat the case $d=1$.

\bigskip
\noindent 
\textbf{Choice of the function $u \in \mathcal{H}_p^1((0, 1) \times \bR_+)$.}
Let $a$ be a constant to be chosen later, such that 
\begin{equation}		\label{240813_eq1}
\frac{3}{2p}-\frac{3}{2}< a<\frac{3}{2p}-1.
\end{equation}
Define 
$u(t,x)=v_x(t,x)$
on $(0,1)\times \bR_+$, where 
$$
v(t,x)=t^{-a} e^{-x^2/t}.
$$
We claim that   
$$
u\in \mathring{\cH}^1_{p}((0,1)\times \bR_+).
$$
To see this, note that 
$$
u(t,x)=-2t^{-a-1} x e^{-x^2/t},
$$
$$
u_x(t,x) = -2t^{-a-1}e^{-x^2/t} + 4t^{-a-2}x^2e^{-x^2/t},
$$
and 
\begin{equation}		\label{240813_eq2}
u_t(t,x)=2(a+1)t^{-a-2}xe^{-x^2/t} - 2 t^{-a-3}x^3e^{-x^2/t}=g_x(t,x), 
\end{equation}
where 
$$
g(t,x):=v_t(t,x)=-at^{-a-1}e^{-x^2/t} + t^{-a-2}x^2e^{-x^2/t}.
$$
For $t\in (0,1)$, we calculate the $L_p$-norms by applying the change of variables $y = x/\sqrt{t}$.
For $u(t, \cdot)$, we have
\begin{align*}
    \|u(t, \cdot)\|^p_{L_p(\bR_+)} &= \int_0^\infty \big|-2t^{-a-1} \sqrt{t}y e^{-y^2}\big|^p \sqrt{t} \,dy \\
    &= 2^p t^{-p(a+1/2)+1/2} \int_0^\infty y^p e^{-py^2}\,dy.
\end{align*}
This yields
\begin{equation}		\label{260210_eq2}
\|u(t, \cdot)\|_{L_p(\bR_+)} \le N_1  t^{-a-\frac{1}{2}+\frac{1}{2p}},
\end{equation}
where $N_1 = N_1(a, p)$.
Similarly, we obtain
$$
\|u_x(t, \cdot)\|_{L_p(\bR_+)}+\|g(t, \cdot)\|_{L_p(\bR_+)} \le N_2 t^{-a-1+\frac{1}{2p}},
$$
and for the time derivative $u_t$,
$$
\|u_t(t, \cdot)\|^p_{L_p(\bR_+)} = t^{-p(a+3/2)+1/2} \int_0^\infty \big|2(a+1)y - 2y^3\big|^p e^{-py^2}\,dy,
$$
which implies
$$
\|u_t(t, \cdot)\|_{L_p(\bR_+)} \ge N_3 t^{-a-\frac{3}{2}+\frac{1}{2p}},
$$
for some positive constant $N_2$ and $N_3$ depending only on $a$ and $p$.
Then by \eqref{240813_eq1}, 
\begin{equation}		\label{250123_eq3}
u,  u_x,  g\in L_{p}((0,1)\times \bR_+)
\end{equation}
and
$$
u_t\notin L_{p}((0,1)\times \bR_+).
$$
On the other hand,  since $u_t=g_x\in L_{p}((\varepsilon, 1)\times \bR_+)$ for all $\varepsilon\in (0, 1)$, we obtain
\begin{equation}		\label{240812_C2}
\int_0^1 \int_{\bR_+} u  \varphi_t \,dx\,dt = \int_0^1 \int_{\bR_+} g \varphi_x \,dx\,dt
\end{equation}
for any $\varphi\in C^\infty_0((0,1)\times \bR_+)$, which shows that
$$
u_t\in \bH^{-1}_{p}((0,1)\times \bR_+).
$$
Together with \eqref{250123_eq3}, this implies that $u\in \cH^1_p((0,1)\times \bR_+)$.

\bigskip
\noindent
\textbf{Zero lateral boundary condition: $u \in \mathring{\cH}_p^1((0, 1) \times \bR_+)$.}
Since $u(t,x)=O(x)$ as $x\to 0^+$, one can apply a standard spatial cut-off argument.
Let $\{ \xi_n \} $ be a sequence of functions in $C^\infty(\bR)$ satisfying
$$
\xi_n(x) =
\begin{cases}
    0 & \text{ for } \, x \le \frac{1}{n} , \\[4pt]
    1 & \text{ for } \, x \ge \frac{2}{n},
\end{cases}
$$
and $|\xi_n'| \le 2n$.
Defining $u^{(n)}(t,x) = u(t,x)\xi_n(x)$, we show that
\[
\|u^{(n)} - u\|_{W^{1,0}_p((0,1)\times \bR_+)} \to 0.
\]
We have $u^{(n)}_x - u_x = u_x(\xi_n - 1) + u\xi_n'$.
Since $\operatorname{supp} \xi_n' \subset [1/n, 2/n]$, we observe that $|\xi_n'(x)| \le 2n \le 4/x$.
Thus, the error term is bounded by
$$
|u(t,x)\xi_n'(x)| \le \frac{4}{x}|u(t,x)| = 8t^{-a-1}e^{-x^2/t} \in L_p((0,1)\times \bR_+),
$$
where the inclusion is due to the condition $a <  3/(2p) - 1$ in \eqref{240813_eq1}.
By the dominated convergence theorem, this implies $\|u\xi_n'\|_{L_p ( (0, 1) \times \bR_+) } \to 0$ and thus $\|u^{(n)} - u\|_{W^{1,0}_p( (0, 1) \times \bR_+)} \to 0$.

For the time derivative, we estimate the $\bH^{-1}_p((0,1)\times \bR_+)$ norm of $u^{(n)}_t - u_t = u_t(\xi_n - 1)$.
We represent this distribution as $\mathcal{G}^{(n)}_x(t,x)$, where
$$
\mathcal{G}^{(n)}(t,x) = -\int_x^\infty u_t(t,y)\big(\xi_n(y) - 1\big)\,dy.
$$
Indeed, for any test function $\varphi \in C_0^\infty((0,1)\times \bR_+)$, applying Fubini's theorem gives
\begin{align*}
-\int_0^1 \int_{\bR_+} \mathcal{G}^{(n)} \varphi_x \,dx\,dt &= \int_0^1 \int_0^\infty u_t(t,y)\big(\xi_n(y) - 1\big) \bigg(\int_0^y \varphi_x(t,x)\,dx\bigg) dy\,dt \\
&= \int_0^1 \int_{\bR_+} u_t(\xi_n - 1) \varphi \,dy\,dt,    
\end{align*}
which confirms $\mathcal{G}^{(n)}_x = u_t(\xi_n - 1)$ in the distribution sense.
Since $\xi_n(y) - 1 \equiv 0$ on $[2/n, \infty)$ by definition, $\mathcal{G}^{(n)}(t,x) \equiv 0$ for $x \ge 2/n$. 
For $x < 2/n$, substituting $u_t(t, y) = g_y(t, y)$ and integrating by parts yields
$$
\mathcal{G}^{(n)}(t,x) = g(t,x)\big(\xi_n(x) - 1\big) + \int_x^{2/n} g(t,y)\xi_n'(y)\,dy.
$$
Since $g \in L_p((0,1)\times \bR_+)$, the $L_p$-norm of the 1st term converges to zero as $n \to \infty$.
For the 2nd term, applying H\"older's inequality gives
\begin{align*}
    \bigg| \int_x^{2/n} g(t,y)\xi_n'(y)\,dy \bigg| &\le 2n \int_{1/n}^{2/n} |g(t,y)|\,dy \\
    &\le 2n \bigg(\frac{1}{n}\bigg)^{1-1/p} \bigg( \int_{1/n}^{2/n} |g(t,y)|^p\,dy \bigg)^{1/p}.
\end{align*}
Then, integrating the $p$-th power over $x \in (0, 2/n)$ yields
$$
\int_0^{2/n} \bigg| \int_x^{2/n} g(t,y)\xi_n'(y)\,dy \bigg|^p dx \le \frac{2}{n} \cdot 2^p n \int_{1/n}^{2/n} |g(t,y)|^p\,dy = 2^{p+1} \int_{1/n}^{2/n} |g(t,y)|^p\,dy.
$$
Integrating over $t \in (0,1)$, the last term tends to zero as $n \to \infty$ by the absolute continuity of the Lebesgue integral, since $g \in L_p((0,1)\times \bR_+)$.
Consequently, $\|\mathcal{G}^{(n)}\|_{L_p((0,1)\times \bR_+)} \to 0$, which implies $\|u^{(n)}_t - u_t\|_{\bH^{-1}_p((0, 1 \times \bR_+) )} \to 0$.
A standard mollification then yields a sequence in $C_0^\infty([0,1]\times \bR_+)$, concluding $u\in \mathring{\cH}^1_p((0,1)\times \bR_+)$.
 
\bigskip
\noindent 
\textbf{Zero extension is impossible for $u$.}
Now we prove
$$
\bar{u}\notin \cH^1_{p}((0,1)\times \bR)
$$
by showing that there are no functions $G$ and $F$ in $L_p((0,1)\times \bR)$ satisfying 
$$
\bar{u}_t=G_x+F \quad \text{on }\, (0, 1)\times \bR.
$$
Here, $\bar{u}$ denotes the zero extension of $u$ with respect to $x$.
Suppose, to the contrary,  that such functions $G$ and $F$ exist.
Then,  for any $\varphi\in C^\infty_0((0,1)\times \bR)$,
\begin{equation}		\label{240812_E2}
\int_0^1 \int_{\bR} \bar{u}\varphi_t\,dx\,dt=\int_0^1\int_{\bR} G \varphi_x\,dx\,dt-\int_0^1\int_{\bR} F\varphi\,dx\,dt. 
\end{equation}
We consider the following two cases: $p = 1$ and $p \in (1, \infty)$.

\bigskip

\begin{enumerate}[1.]
\item
$p=1$.
In this case, we fix $a\in (0, 1/2)$.
Let $\eta$ and $\eta_m$ for $m\in \{2,3,\ldots\}$ be piecewise linear continuous functions on $[0,1]$ defined as follows.
First, we set $\eta$ such that
\[
\eta(t)=
\begin{cases}
    1 & \text{ for }\, t\in \left[0, \frac{1}{2}  \right], \\[5pt]
    -4t+3 & \text{ for }\, t \in \left( \frac{1}{2} , \frac{3}{4}  \right), \\[5pt]
    0 & \text{ for }\, t\in \left[ \frac{3}{4}, 1 \right].
\end{cases}
\]
Then, for each $m \ge 2$, we define $\eta_m$ by
\[
\eta_m(t)=
\begin{cases}
    0 & \text{ for }\,  t\in \left[0, \frac{1}{2m} \right], \\[5pt]
    2mt-1 & \text{ for }\,  t \in \left( \frac{1}{2m}, \frac{1}{m} \right], \\[5pt]
    \eta(t) & \text{ for }\, t \in \left( \frac{1}{m}, 1 \right].
\end{cases}
\]
We also take a sequence $\{\zeta_k\}$ in  $C^\infty(\bR)$ such that
\[
\zeta_k(x) = \begin{cases}
    0 & \text{ for  }\, x\le 0, \\[5pt]
    1 & \text{ for }\,  x \ge \frac{1}{k} .
\end{cases}
\]
For a given $\phi\in C^\infty_0(\bR)$, by setting $\varphi(t,x)=\eta_m(t)\zeta_k(x)\phi(x)$ in \eqref{240812_C2}, we have 
$$
\begin{aligned}
\int_0^1 \int_{\bR_+} u (\eta_m \zeta_k \phi)_t\,dx\,dt
&=\int_0^1 \int_{\bR_+} g (\eta_m \zeta_k \phi)_x\,dx\,dt\\
&=-\int_0^1 \int_{\bR_+} g_x (\eta_m \zeta_k \phi)\,dx\,dt,
\end{aligned}
$$
where the last integral is well-defined because $\eta_m$ has compact support in $(0,1)$ and $g_x\in L_{1}((\varepsilon, 1)\times \bR_+)$ for any $\varepsilon\in (0,1)$.
Letting $k\to \infty$, we obtain
$$
\int_0^1 \int_{\bR_+} u(\eta_m \phi)_t\,dx\,dt=-\int_0^1 \int_{\bR_+} g_x(\eta_m \phi)\,dx\,dt.
$$
Together with  \eqref{240812_E2}, this  yields
\begin{equation}		\label{240812_F1}
\int_0^1 \int_{\bR_+} g_x(\eta_m \phi)\,dx\,dt=-\int_0^1 \int_{\bR} G (\eta_m \phi)_x\,dx\,dt+\int_0^1 \int_{\bR} F (\eta_m \phi)\,dx\,dt,
\end{equation}
where, as $ m \to \infty$, the right-hand side converges to a finite number.
Thus, 
\begin{equation}		\label{230805_eq3}
\lim_{m\to \infty} \int_0^1 \int_{\bR_+} g_x(\eta_m \phi)\,dx\,dt \quad \text{exists}.
\end{equation}
Observe that 
$$
\begin{aligned}
\int_0^1  g_x(\cdot, x) \eta_m\,dt
&=\int_0^1 u_t(\cdot, x) \eta_m \,dt=-\int_0^1 u(\cdot, x) \eta_m'\,dt\\
&=2\int_{1/(2m)}^{1/m} t^{-a-1}x e^{-x^2/t}\eta'_m\,dt+ 2\int_{1/2}^1 t^{-a-1}x e^{-x^2/t}\eta'\,dt.
\end{aligned}
$$
Hence, 
$$
\int_0^1 \int_{\bR_+} g_x (\eta_m \phi)\,dx\,dt=:I_m+J, 
$$
where 
$$
I_m=2\int_{1/(2m)}^{1/m} t^{-a-1} \eta'_m\int_{\bR_+} x e^{-x^2/t}\phi\,dx\,dt,
$$
$$
J=2\int_{1/2}^1 t^{-a-1} \eta' \int_{\bR_+}x e^{-x^2/t}\phi\,dx\,dt.
$$
Using the fact that $\eta_m'(t)=2m$ on $[1/(2m), 1/m]$ and applying the change of variables $x \mapsto \sqrt{t}x$, we obtain
$$
I_m=4m\int^{1/m}_{1/(2m)} t^{-a}\int_{\bR_+} x e^{-x^2} \phi(\sqrt{t}x)\,dx dt.
$$
By choosing $\phi$ to be non-increasing on $\bR_+$ (which implies $\phi(x)\ge0$ on $\bR_+$) and positive near $x=0$, for $x>0$ and $t\in [1/(2m), 1/m]$, we have
$$
\phi(\sqrt{t}x) \ge \phi(x/\sqrt{m})\ge \phi(x).
$$
Hence, 
$$
I_m\ge 4m\int_{1/(2m)}^{1/m}t^{-a} \int_{\bR_+} x e^{-x^2}\phi\,dx\,dt= 4m M  \int_{1/(2m)}^{1/m} t^{-a}\,dt,
$$
where 
$$
0 <  M =   \int_{\bR_+}  x e^{-x^2}\phi\,dx<\infty.
$$
Since $a>0$, we see that $I_m\to \infty$ as $m\to \infty$.
This contradicts \eqref{230805_eq3} because $J$ is finite.

\bigskip

\item
$p\in (1, \infty)$.
In this case, we set 
\begin{equation}		\label{250124_eq1}
a=\frac{1}{p}-1<0.
\end{equation}
Clearly, $a$ satisfies \eqref{240813_eq1}.
Let  $\eta \in L_{p'}((0, \tau))$, and let $\{\eta_m\}$ be a sequence in $C^\infty_0((0,\tau))$ such that 
$$
\eta_m\to \eta \quad \text{in }\, L_{p'}((0,\tau))
$$
and 
$$
\eta_m(t)\to \eta(t) \quad \text{for a.e. }\, t\in (0,\tau),
$$
where $p'$ denotes the H\"older conjugate of $p$, and $\tau\in (0,1]$ is a constant to be chosen later.
As seen in \eqref{240812_F1} and \eqref{230805_eq3}, for a given $\phi\in C^\infty_0(\bR)$, we have 
$$
\lim_{m\to \infty} \int_0^\tau \int_{\bR_+} g_x(\eta_m \phi)\,dx\,dt=\cK,
$$
where 
$$
\cK:=-\int_0^\tau \int_{\bR} G (\eta \phi)_x\,dx\,dt+\int_0^\tau \int_{\bR} F (\eta \phi)\,dx\,dt.
$$
By the explicit formula for $g_x$ in \eqref{240813_eq2}, we obtain
$$
\int_{\bR_+} g_x(t, \cdot)\phi\,dx
=2t^{-a-1}H(t),
$$
where 
$$
H(t)=\int_{\bR_+} \big((a+1)x-x^3\big)e^{-x^2} \phi(\sqrt{t}x)\,dx.
$$
Hence, 
$$
\lim_{m\to \infty} \int_0^\tau t^{-a-1} H \eta_m \,dt=\frac{\cK}{2}.
$$
We fix $\phi$ such that $\phi(0) = 1$.
Since $H$ is continuous on $[0,1]$ and 
$$
H(0)=\int_{\bR_+} \big((a+1)x-x^3\big)e^{-x^2}\,dx = a\int_{\bR_+} x e^{-x^2}\,dx<0,
$$
one can find constants $\nu>0$ and $\tau\in (0,1]$ such that 
$$
H(t)\le -\nu \quad \text{for all }\, t\in [0, \tau].
$$
If $\eta\ge 0$, we may choose $\{\eta_m\}$ so that $\eta_m\ge0$ for all $m$.
Since
\[
t^{-a-1}(-H(t))\eta_m(t)\ge 0 \quad \textrm{ on } \,\, (0, \tau),
\]
by using Fatou's lemma and the fact that $|\cK|\le N \|\eta\|_{L_{p'}((0, \tau))}$,  we obtain 
$$
0\le \int_0^\tau t^{-a-1}(-H)\eta\,dt\le \lim_{m\to \infty} \int_0^\tau t^{-a-1}(-H)\eta_m\,dt\le N \|\eta\|_{L_{p'}((0, \tau))}.
$$
Therefore,  for a general $\eta\in L_{p'}((0, \tau))$, 
$$
\begin{aligned}
\bigg|\int_0^\tau t^{-a-1}H\eta\,dt\bigg|&=\bigg|\int_0^\tau t^{-a-1}H(\eta_+-\eta_-)\,dt\bigg|\\
&\le \int_0^\tau t^{-a-1} (-H) \eta_+\,dt+\int_0^\tau t^{-a-1}(-H)\eta_-\,dt\\
&\le N \|\eta\|_{L_{p'}((0,\tau))},
\end{aligned}
$$
where $\eta_+$ and $\eta_-$ are the positive and negative parts of $\eta$, respectively.
This estimate, together with  duality, implies that $t^{-a-1}H\in L_{p}((0, \tau))$.
Since $|H|$ is bounded away from zero on $[0,\tau]$, we conclude $t^{-a-1}\in L_{p}((0,\tau))$, which contradicts  the condition \eqref{250124_eq1}.
\end{enumerate}
This completes the proof for the case $d=1$.

For $d\ge 2$, we take 
$$
u(t,x)=v_{x_1}(t, x_1)w(x'),
$$
where $x=(x_1, x')\in \bR_+\times \bR^{d-1}=\bR^d_+$, $w\in C^\infty_0(\bR^{d-1})$, and $v$ is 
defined as before.
Then by direct calculation, we see that $u$ provides the desired example for the case $d\ge 2$.
\qed

\subsection{The case $p=\infty$}\label{Sec:p_infty}
To see why the counterexample for $p \in (1, \infty)$ in Theorem~\ref{MT} fails in the limit $p=\infty$, we formally set the exponent \eqref{250124_eq1} to $a = -1$.
The functions $u$ and $g=v_t$ then become
$$
u(t,x) = -2x e^{-x^2/t}, \quad g(t,x) = \Big(1 + \frac{x^2}{t}\Big) e^{-x^2/t}.
$$
Unlike the finite $p$ case where $g(t,0)$ has a non-integrable singularity as $t \to 0^+$, here $g(t,0)$ is bounded.
Thus, the singularity preventing the zero extension does not occur.
In fact, we have the following:

\begin{proposition} \label{prop_inf}
Let $T \in (0, \infty]$. For any $u \in \mathring{\cH}^1_\infty((0,T)\times \bR^d_+)$, its zero extension $\bar{u}$ with respect to $x$ satisfies
$$
\bar{u} \in \cH^1_\infty((0,T)\times \bR^d).
$$
\end{proposition}
\begin{proof}
Since $u \in \mathring{\cH}^1_\infty((0,T)\times \bR^d_+)$, we have $u, Du \in L_\infty((0,T)\times \bR^d_+)$, and there exist $g_1, \ldots, g_d, f \in L_\infty((0,T)\times \bR^d_+)$ such that
\begin{equation}		\label{260616_eq1}
u_t = D_i g_i + f \quad \text{in }\, (0,T)\times \bR^d_+.
\end{equation}
For the zero extension $\bar{u}$, it is clear that $\bar{u}, D\bar{u} \in L_\infty((0,T)\times \bR^d)$.
To show that $\bar{u}_t$ can be represented in divergence form, let $\varphi \in C^\infty_0((0,T)\times \bR^d)$ be an arbitrary test function.
We consider the functional
$$
I = I(\varphi) := -\int_0^T \int_{\bR^d} \bar{u} \varphi_t \,dx\,dt = -\int_0^T \int_{\bR^d_+} u \varphi_t \,dx\,dt.
$$

For $\varepsilon > 0$, we introduce a cut-off function $\zeta_\varepsilon(x_1)$ defined by
\[
\zeta_\varepsilon(x_1) = 
\begin{cases}
    0 & \text{for } x_1 \le 0, \\[4pt]
    x_1/\varepsilon & \text{for } 0 < x_1 < \varepsilon, \\[4pt]
    1 & \text{for } x_1 \ge \varepsilon.
\end{cases}
\]
Then by testing \eqref{260616_eq1} with $\zeta_{\varepsilon}\varphi$, we have
$$
-\int_0^T \int_{\bR^d_+} u (\zeta_\varepsilon \varphi)_t \,dx\,dt = \int_0^T \int_{\bR^d_+} - g_i D_i(\zeta_\varepsilon \varphi) + f \zeta_\varepsilon \varphi  \,dx\,dt.
$$
Taking the limit as $\varepsilon \to 0$, the left-hand side converges to $I$. 
On the right-hand side, since $D_1 \zeta_\varepsilon = \varepsilon^{-1} I_{(0, \varepsilon)}(x_1)$, we obtain
\begin{equation}
    \label{eq:p_inf}
I = \int_0^T \int_{\bR^d_+} - g_i D_i \varphi + f \varphi  \,dx\,dt + \lim_{\varepsilon \to 0^+} \frac{1}{\varepsilon} \int_0^T \int_0^\varepsilon \int_{\bR^{d-1}} g_1 \varphi \,dx'\,dx_1\,dt.
\end{equation}
Note that
$$
|\varphi(t, x_1, x')| = \bigg| \int_{-\infty}^{x_1} D_1 \varphi(t, s, x') \,ds \bigg| \le \int_{-\infty}^\infty |D_1 \varphi(t, s, x')| \,ds.
$$
Using this, we estimate the last term in \eqref{eq:p_inf} as follows:
$$
\begin{aligned}
&\bigg| \frac{1}{\varepsilon} \int_0^T \int_0^\varepsilon \int_{\bR^{d-1}} g_1 \varphi \,dx'\,dx_1\,dt \bigg| \\
&\le \|g_1\|_{L_\infty((0,T)\times \bR^d_+)} \frac{1}{\varepsilon} \int_0^\varepsilon \bigg( \int_0^T \int_{\bR^{d-1}} \int_{-\infty}^\infty |D_1 \varphi(t,s, x')| \,ds \,dx'\,dt \bigg) dx_1 \\
&= \|g_1\|_{L_\infty((0,T)\times \bR^d_+)} \int_0^T \int_{\bR^d} |D_1 \varphi(t, x)| \,dx\,dt.
\end{aligned}
$$
Therefore, 
$$
|I| \le N \int_0^T \int_{\bR^d} \bigg( |\varphi| + \sum_{i=1}^d |D_i \varphi| \bigg) \,dx\,dt,
$$
where
\[
N=N\left(\|g_i\|_{L_\infty((0,T)\times \bR^d_+)}, \|f\|_{L_\infty((0,T)\times \bR^d_+)}\right) > 0.
\]
This inequality implies that the linear mapping
\[
(\varphi, D_1 \varphi, \dots, D_d \varphi) \mapsto I
\]
is bounded with respect to the $L_1$-norm.
By the Hahn-Banach theorem and the duality relation $(L_1)^* = L_\infty$, there exist functions $\bar{F}, \bar{G}_1, \dots, \bar{G}_d \in L_\infty((0,T)\times \bR^d)$ such that
$$
I = \int_0^T \int_{\bR^d} -\bar{G}_i D_i \varphi +  \bar{F} \varphi \,dx\,dt.
$$
This concludes that
\[
\bar{u}_t = D_i \bar{G}_i + \bar{F} \quad \textrm{ in }\,\, (0,T)\times \bR^d.
\]
Hence, $\bar{u} \in \cH^1_\infty((0,T)\times \bR^d)$.
\end{proof}

\bibliographystyle{amsplain}

\def\cprime{$'$}

\end{document}